\documentclass{article}

\usepackage{amscd}
\usepackage{amssymb} 
\usepackage{amsmath} 
\usepackage{latexsym} 
\usepackage{theorem} 
\usepackage{color}
\usepackage{epsfig}

\usepackage{makeidx}

\usepackage{eepic,slashbox}

\theorembodyfont{\itshape} 

\newtheorem{theorem}{Theorem}
\newtheorem{lemma}[theorem]{Lemma}

\newtheorem{corollary}[theorem]{Corollary}
\newtheorem{proposition}[theorem]{Proposition}

\newtheorem{definition}[theorem]{Definition}

\newtheorem{remark}[theorem]{Remark}
\newenvironment{proof}{{\par\addvspace{0.1cm}\noindent \bf Proof. }}{\hfill$\Box$\par\medskip} 
 
\def\R{\Re\mathfrak{e} \,}
\def\I{\Im\mathfrak{m} \,}

\title{\bf Equisectional equivalence of triangles}
\author{Jun O'Hara}

\begin{document}

\maketitle

\begin{abstract} 
We study equivalence relation of the set of triangles generated by similarity and operation on a triangle to get a new one by joining division points of three edges with the same ratio. 
Using the moduli space of similarity classes of triangles introduced by Nakamura and Oguiso, we give characterization of equivalent triangles in terms of circles of Apollonius (or hyperbolic pencil of circles) and properties of special equivalent triangles. We also study rationality problem and constructibility problem. 
\end{abstract}

\medskip
{\small {\it Key words and phrases}. Triangle, moduli space, circle of Apollonius, hyperbolic pencil of circles. 
}

{\small 2010 {\it Mathematics Subject Classification}: 51M04.}

\section{Introduction}
We study elementary geometric operations on triangles defined as follows. 
Let $\triangle ABC$ be a triangle, and $q$ be a real number. 
Let $A',B'$, and $C'$ be division points of the edges $BC,CA$, and $AB$ by $1-q:q$ respectively, namely, \setlength\arraycolsep{1pt}
\[
A'=qB+(1-q)C,\>\>\>
B'=qC+(1-q)A,\>\>\>
C'=qA+(1-q)B.
\] 
Let $A''$ ($B''$ or $C''$) be the intersection of the lines $AA'$ and $BB'$ ($BB'$ and $CC'$ or $CC'$ and $AA'$ respectively). 
Define {\em equisection operators} 
$T_q$ and $S_q$, where $S_q$ can be defined when $q\ne1/2$, by 
\[T_q(\triangle ABC\,)=\triangle A'B'C' \>\>\mbox{ and }\>\> S_q(\triangle ABC\,)=\triangle A''B''C''\,.\]
The operators $T_q$ have been studied in articles such as \cite{CD, D, NO, S}, {\em et. al.}

\begin{figure}[htbp]
\begin{center}
\begin{minipage}{.45\linewidth}
\begin{center}
\includegraphics[width=0.8\linewidth]{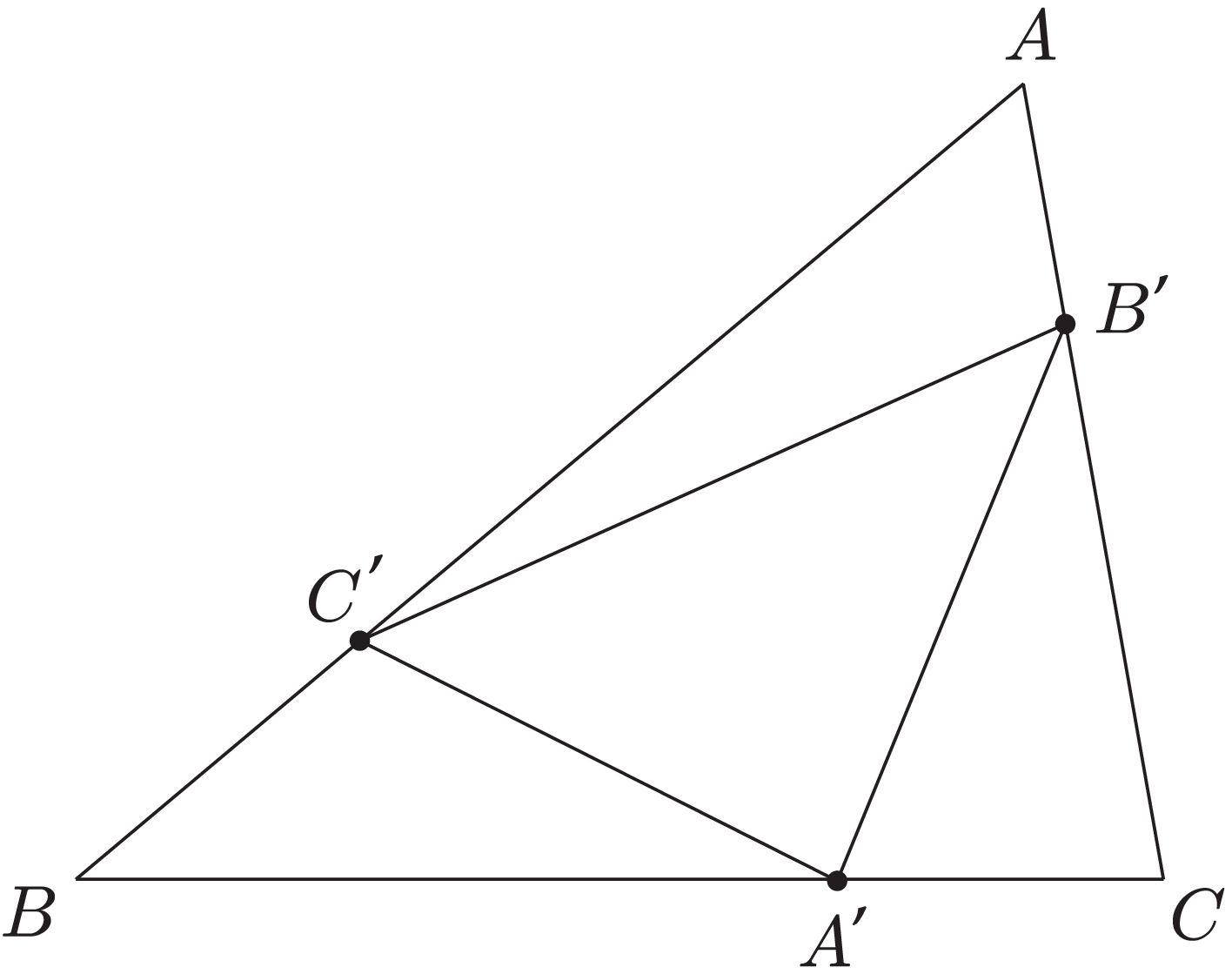}
\caption{Operator $T_q$}
\label{fig_T_q}
\end{center}
\end{minipage}
\hskip 0.4cm
\begin{minipage}{.45\linewidth}
\begin{center}
\includegraphics[width=0.8\linewidth]{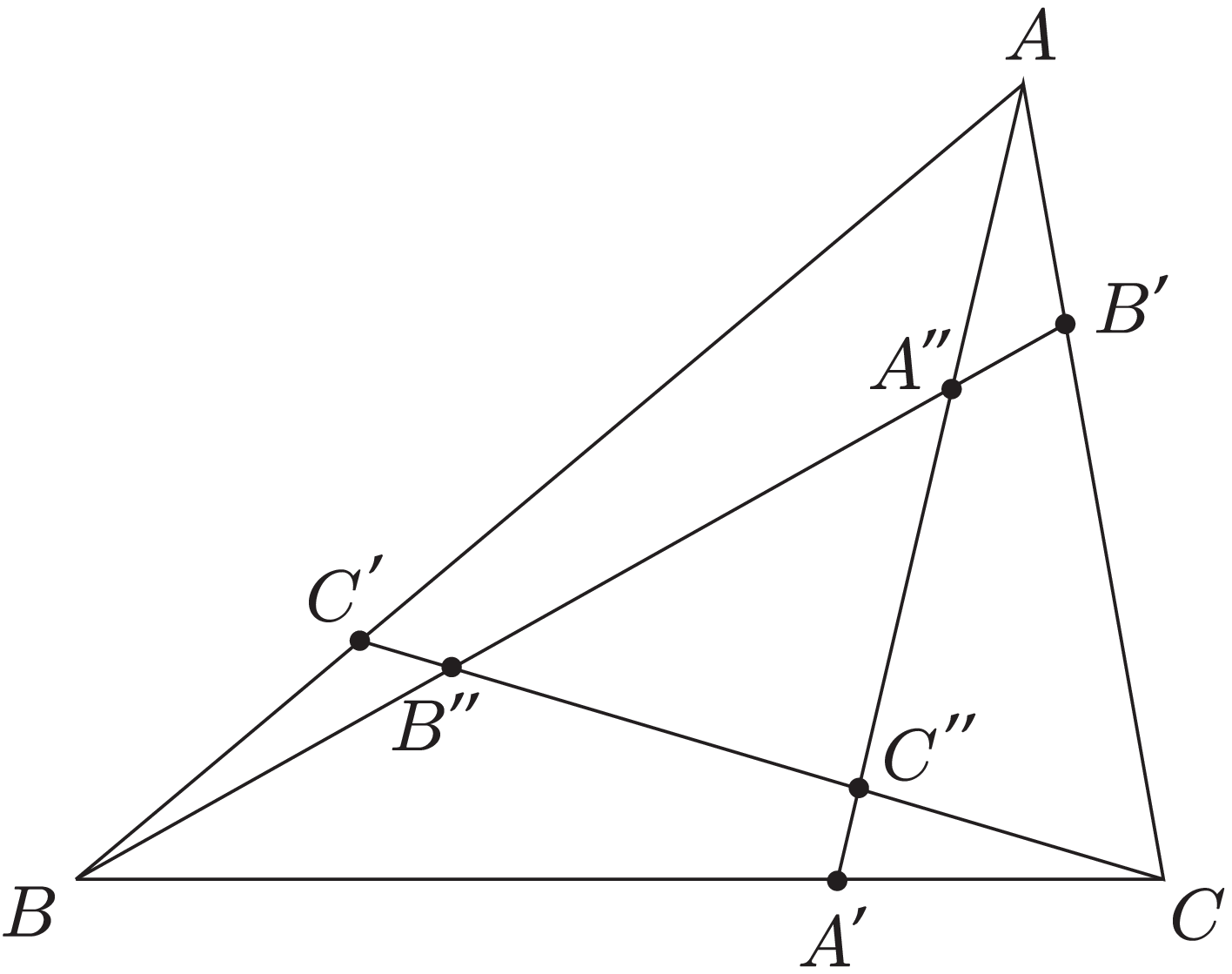}
\caption{Operator $S_q$}
\label{fig_S_q}
\end{center}
\end{minipage}
\end{center}
\end{figure}

In this note we study the equivalence relation (denoted by $\sim$) of the set of triangles (denoted by $\mathcal{T}$) generated by similarity and $\{T_q\}_{q\in \mathbb{R}}$, which we shall call {\em equisectional equivalence}. 
The equivalence relation generated by similarity and $\{T_q\}_{q\in \mathbb{Q}}$ shall be called {\em rational equisectional equivalence} and denoted by $\sim_\mathbb{Q}$. 
We say two triangles $\Delta$ and $\Delta'$ are {\em equisectionally equivalent} (or {\em rational equisectionally equivalent}) if $\Delta \sim \Delta'$ (or $\Delta \sim_\mathbb{Q} \Delta'$ respectively). 
We remark that we use the term  ``similarity'' as the equivalence under orientatipon preserving homothetic transformation in this article. We say two triangles are reversely similar if they are equivalent under orientation reversing homothetic transformation. 

Nakamura and Oguiso introduced the moduli space of similarity classes of triangles in \cite{NO}, which is a strong tool for the study of $T_q$ and $S_q$. 
Using their results (explained in Section \ref{section_NO}), we give (projective) geometric characterization of equisectionally equivalent triangles. 
Namely, two triangles with a common base, say $BC$, with the third vertices, say $A$ and $\widehat A'$, in the same side of the base are equisectionally equivalent if and only if $A$ and $\widehat A'$ 
are on the same circle of Apollonius with foci being two vertices (denoted by $D$ and $D'$) of regular triangles with the common base $BC$. 
Therefore, each equisectional equivalence class with a given base $BC$ corresponds to a circle of Apollonius with foci $D$ and $D'$. It is an element of a hyperbolic pencil of circles defined by $D$ and $D'$ from a projective geometric viewpoint. 

We then study properties of triangles of the following three special types, right triangles, isosceles triangles, and trianges with sides in arithmetic progression (which shall be denoted by {\em SAP} triangles), that appear in the same equisectional equivalence class. There are (at most) two similarity classes of such triangles for each type, which are reversely similar in the case of right or SAP triangles, or the base angles of which satisfy $\tan\theta\cdot\tan\theta'=3$ in the case of isosceles triangles. For each type we explicitly give the ratio $q$ such that $T_q$ maps one to the other in the same equisectional equivalence class, which implies that a pair of triangles $\Delta$ and $\Delta'$ of one of the above special types with rational edges satisfies $\Delta\sim\Delta'$ if and only if $\Delta\sim_\mathbb{Q}\Delta'$. 

We finally study compass and straightedge constructibility of $q$ for a given pair of triangles. 

\section{The statement of the main results}

\begin{definition}\label{def_alpha} \rm 
Let $\Delta=\triangle ABC$ be a triangle. 
Let $\Pi_A$ be a half plane containing $A$ with boundary the line $\overline{BC}$, and 
$D$ and $\overline D$ be two points ($D\in\Pi_A$) such that $\triangle DBC$ and $\triangle \overline DBC$ are regular triangles. 
Define $\alpha(\Delta)$ ($0\le\alpha(\Delta)<1$) and $\beta(\Delta)$ by 
\[
\alpha(\Delta)=\frac{|AD|}{\big|A\overline{D}\big|}, \quad 
\beta(\Delta)=\frac{|\triangle ABC|}{\big|A\overline{D}\big|^2}\,,
\]
where $|\triangle ABC|$ means the area of $\triangle ABC$.
\end{definition}

We remark that both $\alpha(\Delta)$ and $\beta(\Delta)$ are independent of the choice of the base of the triangle. 
A locus of points $X$ such that $|XD|/\big|X\overline{D}\big|$ is a given positive constant 
is a circle, called a {\em circle of Apollonius with foci $D$ and $\overline D$}. 
Put 
\[
C_A=\{X\,:\,|XD|/\big|X\overline{D}\big|=\alpha(\Delta)\}. 
\]
Note that $C_A=\{A\}$ when $\Delta$ is a regular triangle. 
The quantity $\alpha$ takes the value $0$ if and only if $\Delta$ is a regular triangle, and approaches $1$ as $\Delta$ becomes thinner and thinner. In that sense, it can be considered as measuring how far a triangle is from a regular triangle. 

\begin{theorem}\label{main_theorem} 
Given two triangles $\Delta=\triangle ABC$ and $\Delta'=\triangle A'BC$. 
Let $\widehat A'$ be a point in $\Pi_A$ such that $\triangle \widehat A'BC$ is similar to $\triangle A'B'C'$. 
Then the following conditions are equivalent:
\begin{enumerate}
\item $\Delta$ is equisectionally equivalent to $\Delta'\,$.
\item $\alpha(\Delta)=\alpha(\Delta')$, in other words, $\widehat A'$ is on the circle of Apollonius with foci $D$ and $\overline{D}$ that passes through $A$. 
\item $\beta(\Delta)=\beta(\Delta')$. 
\item Let $A_2$ and $A_3$ be points in $\Pi_A$ such that $\triangle BCA_2$ and $\triangle CA_3B$ are similar to $\triangle ABC$ 
in such a way that each vertex of $\triangle BCA_2$ or $\triangle CA_3B$ corresponds to a vertex of $\triangle ABC$ in the same sequential order through the similarity (Figure \ref{fig_Ap_circ_three_pts}). 
Then $\widehat A'$ is on the circle that passes through $A, A_2$, and $A_3$. 
When $\Delta$ is a regular triangle we agree that the circle through $A, A_2$, and $A_3$ consists of a single point. 
\end{enumerate}
\end{theorem}

\begin{figure}[htbp]
\begin{center}
\begin{minipage}{.4\linewidth}
\begin{center}
\includegraphics[width=0.9\linewidth]{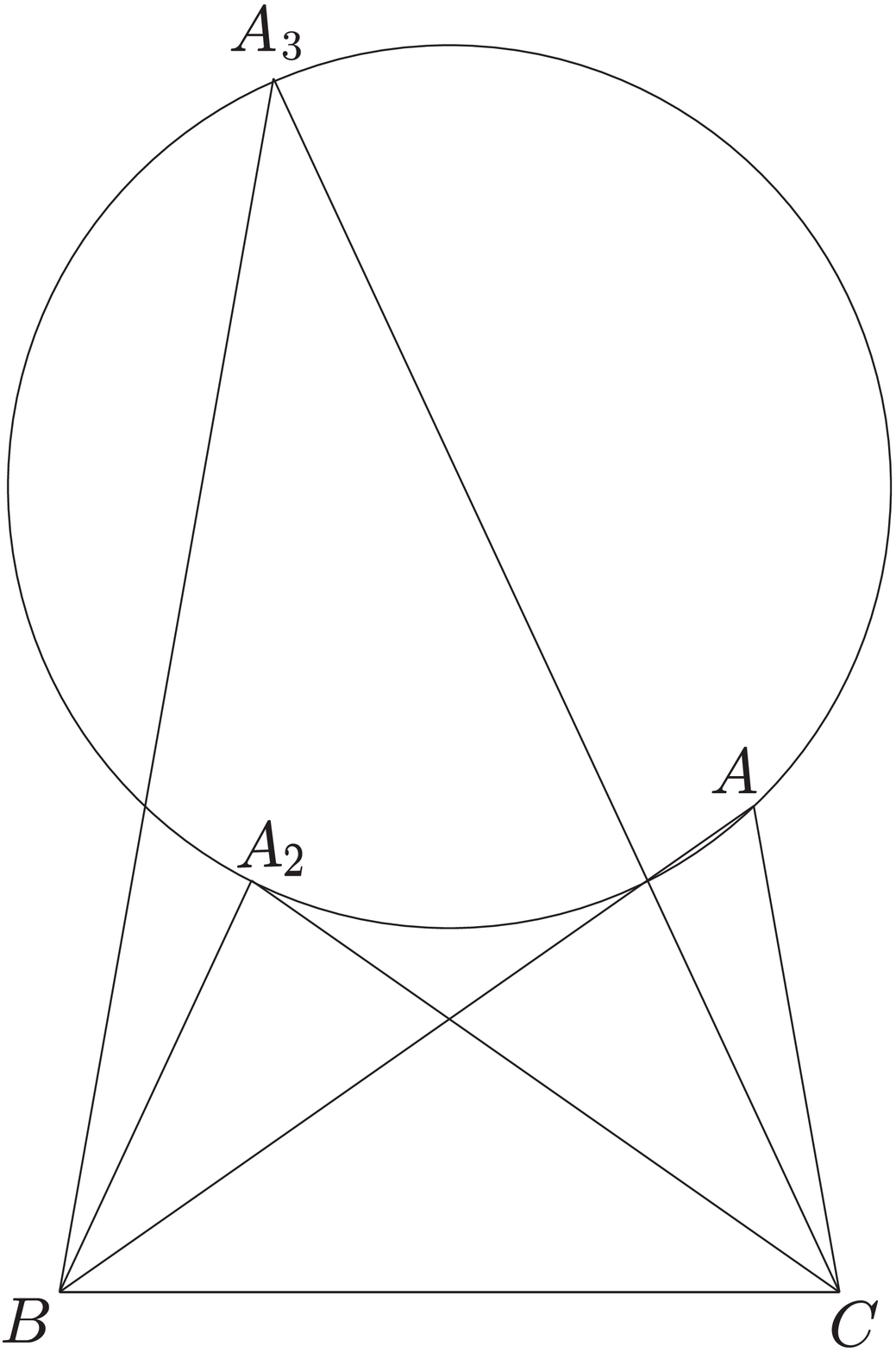}
\caption{Three positions of similar triangles with a common base in the same half-plane.}
\label{fig_Ap_circ_three_pts}
\end{center}
\end{minipage}
\hskip 0.8cm
\begin{minipage}{.5\linewidth}
\begin{center}
\includegraphics[width=0.9\linewidth]{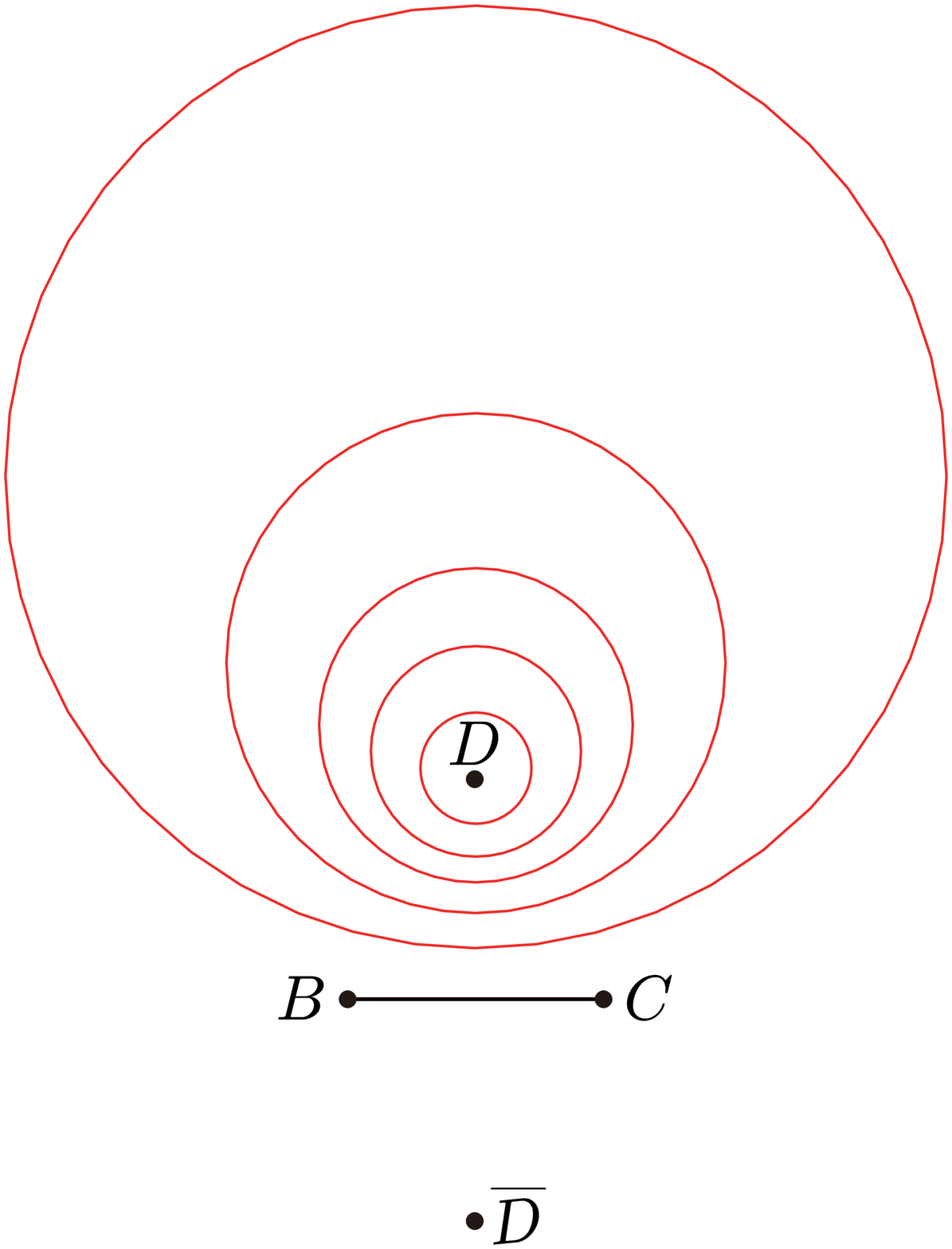}
\caption{A hyperbolic pencil consisting of circles of Apollonius with foci $D$ and $\overline{D}$}
\label{fig_Apollonius_circles}
\end{center}
\end{minipage}
\end{center}
\end{figure}

The set of circles of Apollonius with foci $D$ and $\overline{D}$ is called a {\em hyperbolic pencil} of circles defined by $D$ and $\overline D$ (or a {\em Poncelet pencil} with {\em limit points} (or {\em Poncelet points}) $D$ and $\overline D$). It consists of circles that are orthogonal to any circle passing through $D$ and $\overline D$ (Figure \ref{pencil} left). A set of circles through $D$ and $\overline D$ is called an {\em elliptic pencil} (or a {\em pencil of circles with base points}).

\smallskip
Let $\mathcal{S}$ be the set of similarity classes of triangles and $[\Delta]$ denote the similarity class of a triangle $\Delta$. 
Nakamura and Oguiso's result implies that the sets of similarity classes of equisectionally equivalent triangles form a codimension $1$ foliation of $\mathcal{S}$ with a unique singularity $0$ that corresponds to regular triangles. We study the intersection of each leaf and another codimension $1$ subspace of $\mathcal{S}$ which is the set of similarity classes of one of the following three special triangles, isosceles triangles, right triangles, and SAP triangles (i.e., triangles with sides in arithmetic progression) (the reader is referred to \cite{MOS} for the properties of SAP triangles). 

\begin{corollary}\label{cor_right}
Two right triangles are equisectionally equivalent if and only if they are either similar or reversely similar. 
Any triangle $\Delta$ is equisectionally equivalent to a right triangle if and only if $\alpha(\Delta)\ge(\sqrt3-1)^2/2$. 
\end{corollary}

\begin{corollary}\label{cor_SAP}
Two SAP triangles are equisectionally equivalent if and only if they are either similar or reversely similar. 
Any trianle is equisectionally equivalent to such a triangle. 
\end{corollary}

\begin{proposition}\label{prop_isosceles}
Two isosceles triangles $\Delta$ and $\Delta'$ are equisectionally equivalent if and only if either they are similar or the base angles satisfy $\tan\angle \theta \cdot\tan\angle \theta'=3$. 
When $[\Delta]\ne[\Delta']$, $[T_q(\Delta)]=[\Delta']$ if and only if $q=1/3$ or $2/3$. Any triangle is equisectionally equivalent to an isosceles triangle. 

\end{proposition}

Corollaries \ref{cor_right} and \ref{cor_SAP} imply the importance of an equisectional operator that maps a similarity class of a given triangle to that of its mirror image. 
\begin{lemma}\label{lemma}
Let $\Delta_{ab}$ denote a triangle with side lengths $1,a,b$ in the anti-clockwise order. Then $[T_q(\Delta_{ab})]=[\Delta_{ba}]$ if and only if 
\[
q=\frac{b^2-1}{a^2+b^2-2},\>\> \frac{1-a^2}{b^2+1-2a^2},\>\>\mbox{ or }\>\> \frac{a^2-b^2}{1+a^2-2b^2}\,.
\]
\end{lemma}
\begin{theorem}\label{thm_rationality}
Suppose $\Delta$ and $\Delta'$ are isosceles (or right or SAP) triangles such that all the sides are rational numbers. If $T_q(\Delta)=\Delta'$ then $q$ is also rational, namely, $\Delta\sim\Delta'$ if and only if $\Delta\sim_\mathbb{Q}\Delta'$. 
\end{theorem}

\section{The moduli space by Nakamura and Oguiso}\label{section_NO}
%
Nakamura and Oguiso gave a bijection between $\mathcal{S}$ and the open unit disc $\mathcal{D}$ in $\mathbb{C}$, and showed that the set of equisection operators $\{T_q\}_{q\in\mathbb{R}}$ acts on $\mathcal{D}$ as rotations. 

\smallskip
Let us first introduce the result of Nakamura-Oguiso \cite{NO}. 
We work in the complex plane $\mathbb{C}$. Let $\mathcal{H}$ be the upper half plane $\{z\in\mathbb{C}\,:\,\I z>0\}$, $\mathcal{B}$ and $\mathcal{D}$ be the open unit discs in variables $Z$ and $w$ respectively. Put $\rho=\exp(\pi i/3)$. Define $\lambda\colon\mathcal{H}\to\mathcal{H}, \, \mu\colon\mathcal{B}\to\mathcal{B}, \, f\colon\mathcal{H}\to\mathcal{B}$, and $g\colon\mathcal{B}\to\mathcal{D}$ by 
\begin{equation}\label{def_NO}
\lambda(z)=\frac1{1-z}, \quad \mu(Z)=\rho^2Z, \quad f(z)=\frac{\rho^2-\rho z}{\rho^2+z}=-\rho\,\frac{z-\rho}{z-\rho^{-1}}, \quad g(Z)=Z^3.
\end{equation}
Then we have the following commutative diagram: 

$$
\begin{CD}
\mathcal{H} @>f >> \mathcal{B} @>g >> \mathcal{D}\\
@V{\lambda}VV @VV{\mu}V  @| \\
\mathcal{H} @>f >> \mathcal{B} @>g >> \mathcal{D}
\end{CD}
$$

Let us fix the base of a triangle to be $[01]$ so that a triangle can be identified by the vertex $z\in\mathcal{H}$. 
Suppose $\Delta=\triangle ABC$ is similar to the triangle $\triangle z01$. Since the three choices of the base, $BC, CA$, or $AB$ corresponds to $z, \lambda(z)$, or $\lambda^2(z)$, there is a bijection $\varphi\colon\mathcal{S}\to\mathcal{D}$ (\cite{NO}) given by 
\begin{equation}\label{def_varphi}
\varphi([\triangle z01])=g\circ f(z)=\left(\frac{z-\rho}{z-\rho^{-1}}\right)^3\,.
\end{equation}
Let us call $\mathcal{D}$ the {\em Nakamura-Oguiso moduli space} of the similarity classes of triangles. 

From the construction of the moduli space and the property of linear fractional transformations, it follows that the set of isosceles triangles is expressed by a real axis in $\mathcal{D}$ (explained in Section \ref{section_proofs}), and reversely similar triangles by complex conjugate numbers, as was pointed out in \cite{NO}. 

The equisection operators and the equisectional equivalence relation on $\mathcal{S}$, denoted by the same symbols, can be naturally induced from those on $\mathcal{T}$. 
We shall express the operators on $\mathcal{D}$ given by $\varphi\circ T_q\circ\varphi^{-1}$ and $\varphi\circ S_q\circ\varphi^{-1}$ simply by $T_q$ and $S_q$ respectively. 
Since our $T_q$ and $S_q$ are equal to $T_{0,q}$ and $T_{1-q,q}$ in \cite{NO} respectively, Theorem 1 of \cite{NO} implies
\begin{theorem}\label{thm_NO}{\rm (\cite{NO})}
The operator $T_q$ acts on $\mathcal{D}$ as a rotation by angle 
\begin{eqnarray}\label{formula_T_q}
6\arg\left((1-2q)\rho-(1-q)\right)=6\arg \left(-1+(1-2q)\sqrt3\,i\right),
\end{eqnarray}
and $S_q$ by angle $6\arg(\rho-q)$. 
\end{theorem}

Theorem \ref{thm_NO} implies that $\sim$ is in fact an equivalence relation, that $\Delta\sim\Delta'$ if and only if there is a real numbers $q$ such that $\Delta'$ is similar to $T_q(\Delta)$, and that the equivalence relation generated by similarity and $\{S_q\}_{q\in \mathbb{R}\setminus\{1/2\}}$ is identical with the equisectional equivalence.

\bigskip
The following corollaries can be obtained by simple computation. 

\begin{corollary}\label{cor_T_q_id}{\rm (\cite{NO})} \begin{enumerate}
\item $T_q=Id_{\mathcal{D}}$ if and only if $q=0,1/2$, or $1$. 
\item $T_q\ne Id_{\mathcal{D}}$ and $T_q{}^2=Id_{\mathcal{D}}$ if and only if $q=1/3$ or $2/3$. 
\end{enumerate}
\end{corollary}

\begin{corollary}\label{cor_T_q_equal}
$T_{q'}=T_q$ if and only if 
\[q'=q,\>\>\>q'=\frac{2q-1}{3q-1} \>\>\>\left(q\ne\frac13\right), \>\>\mbox{ or }\>\>\>q'=\frac{q-1}{3q-2}\>\>\>\left(q\ne\frac23\right).
\]
\end{corollary}

\begin{corollary}\label{cor_T_q_inverse}{\rm (\cite{NO})}
$T_{q'}=T_q{}^{-1}$ if and only if 
\[q'=1-q,\>\>\>q'=\frac{2q-1}{3q-2} \>\>\>\left(q\ne\frac23\right), \>\>\mbox{ or }\>\>\>q'=\frac{q}{3q-1}\>\>\>\left(q\ne\frac13\right).
\]
\end{corollary}

\begin{remark}\rm The six functions of $q$ which appear in the right hand sides in Corollaries \ref{cor_T_q_equal} and \ref{cor_T_q_inverse} form a non-abelian group with the operation of composition, which is isomorphic to the full permutation group of three elements. 
\end{remark}

\begin{corollary}\label{cor_T_q_composition} 
Given real numbers $q$ and $q'$. $T_q\circ T_{q'}=T_{q'}\circ T_q=T_{q''}$ if and only if 
\[
q''=\frac{3qq'-2(q+q')+1}{6qq'-3(q+q')+1}, \>\>\>
q''=-\frac{q+q'-1}{3qq'-3(q+q')+2}, \>\>\mbox{ or }\>\>\>
q''=\frac{3qq'-(q+q')}{3qq'-1}.
\]
\end{corollary}

\begin{proof} The formula \eqref{formula_T_q} implies that $T_q\circ T_{q'}$ is the rotation by angle 
\[
6\arg\left[\,(-1+(q+q'))\rho+(q+q')-3qq'\,\right].
\]
Since 
\[
6\arg(A\rho+B)\equiv 6\arg\left[\left(1-2\frac{A+B}{A+2B}\right)\rho-\left(1-\frac{A+B}{A+2B}\right)\right] \hspace{0.7cm}(\mbox{modulo }\>2\pi),
\]
substitution $A=(q+q')-1$ and $B=-3qq'+(q+q')$ implies that if we put 
\[
q''=\frac{A+B}{A+2B}=\frac{3qq'-2(q+q')+1}{6qq'-3(q+q')+1},
\]
then $T_q\circ T_{q'}=T_{q'}\circ T_q=T_{q''}$. 
The other two values for $q''$ can be obtained by Corollary \ref{cor_T_q_equal}. 
\end{proof}

Corollaries \ref{cor_T_q_id} (1), \ref{cor_T_q_inverse}, and \ref{cor_T_q_composition} show that the rational equisectional equivalence $\sim_\mathbb{Q}$ is in fact an equivalence relation. 

\section{Proofs of the main results}\label{section_proofs}
In what follows, we restrict ourselves to the case of non-regular triangles. 

\bigskip
\noindent
{\bfseries Proof of Theorem {\rm \bf \ref{main_theorem}.}} \ \ 
Theorem \ref{thm_NO} shows that a set of equisectionally equivalent triangles corresponds to a circle with center $0$ in $\mathcal{D}$. 
Therefore, the formula \eqref{def_varphi} implies 
\[
[\triangle z01]\sim[\triangle z'01]\Longleftrightarrow
\left|\left(\frac{z-\rho}{z-\rho^{-1}}\right)^3\right|=\left|\left(\frac{z'-\rho}{z'-\rho^{-1}}\right)^3\right|
\Longleftrightarrow
\frac{|z-\rho|}{|z-\rho^{-1}|}=\frac{|z'-\rho|}{|z'-\rho^{-1}|},
\]
which proves the equivalence between (1) and (2).

\smallskip
Since $\triangle ABC$, $\triangle A_2BC$, and $\triangle A_3BC$ are similar, the above argument implies that a circle of Apollonius with foci $\rho$ and $\rho^{-1}$ that passes through $A$ also passes through $A_2$ and $A_3$, which proves the equivalence of (1) and (4). 

\smallskip
The equivalence between (2) and (3) follows from computation. Since 
\begin{equation}\label{f_for_beta}
\frac{|z-\rho|^2}{|z-\rho^{-1}|^2}=1-\frac{2\sqrt3\,\I z}{|z-\rho^{-1}|^2}, 
\end{equation}
\[
\frac{|z-\rho|}{|z-\rho^{-1}|}=\frac{|z'-\rho|}{|z'-\rho^{-1}|}\quad\mbox{ if and only if }\quad \frac{\I z}{|z-\rho^{-1}|^2}=\frac{\I z'}{|z'-\rho^{-1}|^2}, 
\]
which, translated to a scale-invariant statement, is equivalent to (3). 
We remark that the formula \eqref{f_for_beta} implies $\displaystyle \beta(\Delta)=\frac{1-\alpha(\Delta)^2}{4\sqrt3}$
. 
{\hfill{\small{${\square}$}}\par\medskip}

\bigskip
Let us give projective geometric explanation of the equivalence between (1) and (2). 
The set of lines through $0$, which are considered as circles through $0$ and $\infty$, is an elliptic pencil of circles defined by $0$ and $\infty$, and the set of concentric circles with center $0$ is a hyperbolic pencil of circles defined by $0$ and $\infty$. They are mutually orthogonal. 
A linear fractional transformation is a conformal map (i.e., it preserves the angles) that maps circles (which include lines that can be considered as circles through $\infty$) to circles, and hence it maps an elliptic pencil (or a hyperbolic pencil) of circles defined by a pair of points to an elliptic pencil (or a hyperbolic pencil) defined by a pair of corresponding points. 

Since $f^{-1}$ is a linear fractional transformation which maps $0$ and $\infty$ to $\rho$ and $\rho^{-1}$, it maps the set of lines through $0$ to an elliptic pencil consisting of the circles through $\rho$ and $\rho^{-1}$, and the set of concentric circles with center $0$ to a hyperbolic pencil defined by $\rho$ and $\rho^{-1}$, which consists of circles of Apollonius with foci $\rho$ and $\rho^{-1}$ (Figure \ref{pencil}). 
Now it follows from the construction of Nakamura-Oguiso moduli space that the vertices $z$ of equisectionally equivalent triangles $\triangle z01$ form a circle of Apollonius with foci $\rho$ and $\rho^{-1}$, which proves the equivalence between (1) and (2). 

\begin{figure}[htbp]
\begin{center}
\includegraphics[width=.8\linewidth]{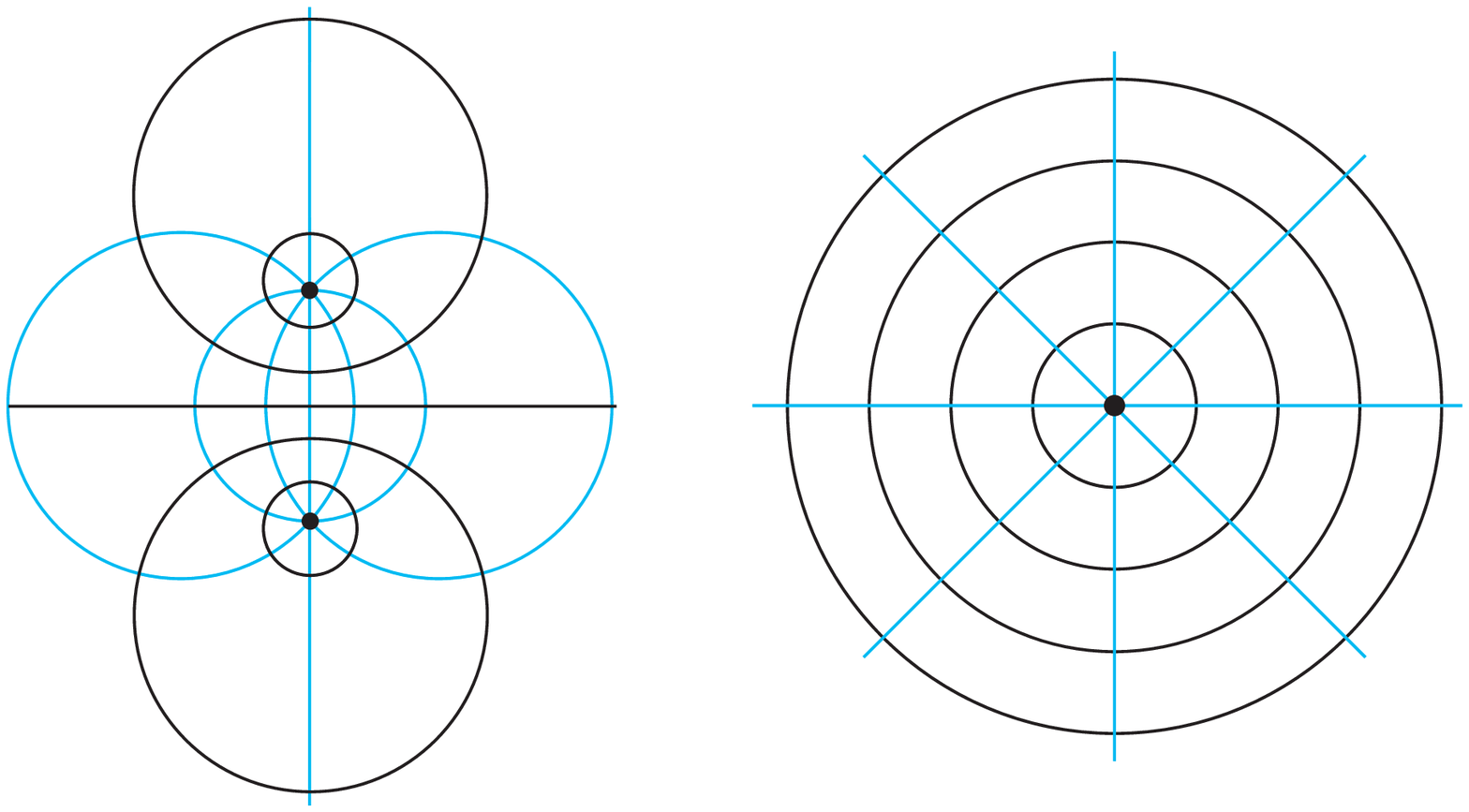}
\caption{A pair of a hyperbolic pencil (left, black) and an elliptic pencil (left, blue) of circles defined by $\rho$ and $\rho^{-1}$, and a pair of a hyperbolic pencil (right, black) and an elliptic pencil (right, blue) of circles defined by $0$ and $\infty$. The former is mapped to the latter by a linear fractional transformation that maps $\rho$ and $\rho^{-1}$ to $0$ and $\infty$ respectively. In each case, a circle in a hyperbolic pencil is orthogonal to a circle in an elliptic pencil.}
\label{pencil}
\end{center}
\end{figure}

\bigskip
In what follows, it sometimes makes things easier to work in a a fundamental domain of $\lambda$, 
\begin{equation}\label{fundamental_domain}
\Omega=\{z\in\mathcal{H}\,:\,|z|<1,\,|z-1|\le1\}\cup\{\rho\},
\end{equation}
which corresponds to studying triangles with the longest edge (one of the longest edges) being fixed to $[0,1]$. 

Let us explain why the isosceles triangles corresponds to $(-1,1)$ in the real axis in the Nakamura-Oguiso moduli space $\mathcal{D}$. 

\begin{definition}\label{def_Gamma} \rm 
Let $\Gamma(\xi,\eta,\zeta)$ be a circle through three points $\xi,\eta$, and $\zeta$. When one of the three points is $\infty$, $\Gamma(\xi,\eta,\zeta)$ is a line. 
We assume that $\Gamma(\xi,\eta,\zeta)$ is oriented by the cyclic order of $\xi,\eta$, and $\zeta$. 
\end{definition}

In $\Omega$, a vertex of an isosceles triangle lies either on the line $\R z=1/2$ or on the circle $|z-1|=1$. 
Since $f$ is a linear fractional transformation with 
\[
f\colon \rho\mapsto0, \>0\mapsto1, \> 1\mapsto\rho^2, \> \rho^{-1}\mapsto\infty, \>\>\mbox{ and }\>\>\infty\mapsto-\rho,
\]
it maps a circle $|z-1|=1$ ($\Gamma(\rho,0,\rho^{-1})$) to the real axis ($\Gamma(0,1,\infty)$), another circle $|z|=1$ ($\Gamma(\rho,1,\rho^{-1})$) to a line joining $0$ and $\rho^2$ ($\Gamma(0,\rho^2,\infty)$), a line $\R z=1/2$ ($\Gamma(\rho,\rho^{-1},\infty)$) to a line segment joining $0$ and $\rho$ ($\Gamma(0,\infty,-\rho)$), the real axis ($\Gamma(0,1,\infty)$) to the unit circle ($\Gamma(1,\rho^2,-\rho)$), and $\Omega$ to one third of the open unit disc $\{Z\in\mathcal{B}\,:\,|Z|<1, 0\le\textrm{arg}Z<\pi/3\}$. 
Therefore, the images of $\varphi$ of $\Omega\cap\{z\,:\,|z-1|=1\}$ and $\Omega\cap\{z\,:\,\R z=1/2\}$ are $[0,1)$ and $(-1,0]$ respectively. 

\bigskip
\noindent
{\bfseries Proof of Proposition {\rm \bf \ref{prop_isosceles}.}} \ \ 
We work in the Nakamura-Oguis moduli space $\mathcal{D}$. 
The isosceles triangles correspond to the real axis in $\mathcal{D}$. 
Each circle with center $0$, which corresponds to a set of equisectionally equivalent non-regular triangles, intersects the real axis in two points. 
It proves the third statement. 
It also proves that if $T_q$ satisfies $T_q(\Delta)=\Delta'$ with two isosceles triangles $\Delta$ and $\Delta'$ then either $[\Delta]=[\Delta']$ or $T_q$ is a rotation by $\pi$ on $\mathcal{D}$. 

Theorem \ref{thm_NO} shows that $T_q$ is a rotation by angle $(2n+1)\pi$ for some $n\in\mathbb{Z}$ if and only if $q=1/3$ or $2/3$, which proves the second statement. 

The equation $\tan\angle \theta \cdot\tan\angle \theta'=3$ follows directly from the fact that $[\Delta']=T_{1/3}([\Delta])$ as is illustrated in Figure \ref{isosceles3}. 
{\hfill{\small{${\square}$}}\par\medskip}

\begin{figure}[htbp]
\begin{center}
\includegraphics[width=.4\linewidth]{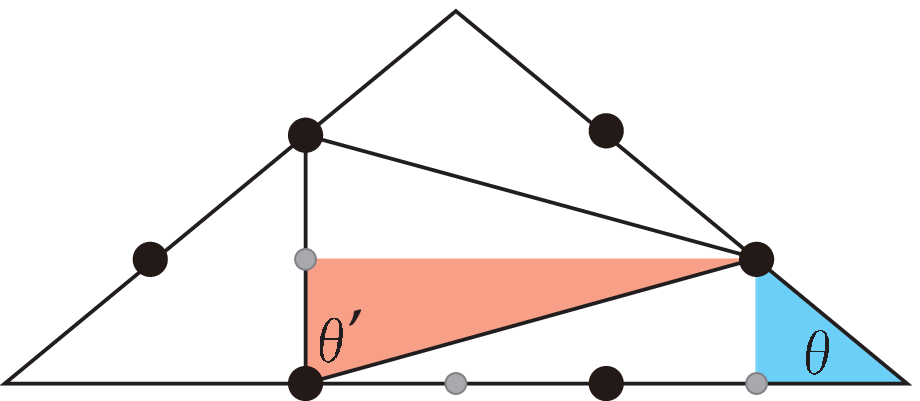}
\caption{$\tan\theta_-\cdot\tan\theta_0=3$}
\label{isosceles3}
\end{center}
\end{figure}

\bigskip
\noindent
{\bfseries Proof of Corollary {\rm \bf \ref{cor_right}.}} \ \ 
In the fundamental domain $\Omega$, a vertex $z$ of a right triangle $\triangle z01$ lies on the upper half hemi-circle $|z-1/2|=1/2$, which intersects any circle of Apollonius in at most two points, which are symmetric in the line $\{x=1/2\}$. The extremal value of $\alpha$ is given by a right isosceles triangle. 
{\hfill{\small{${\square}$}}\par\medskip}

\bigskip
\noindent
{\bfseries Proof of Corollary {\rm \bf \ref{cor_SAP}.}} \ \ 
In the fundamental domain $\Omega$, a non-regular triangle such that the ratio of the edge lengths is $1:1-d:1-2d$ corresponds to 
\[
P_\pm(d)=\left(\frac{1\pm(2d-3d^2)}2,\,\frac{\sqrt3}2\sqrt{(1-d)(1-d^2)(1-3d)}\right)
\quad \left(0<d<\frac13\right).
\]
Let $\Gamma_\pm$ be a curve $\Gamma_\pm=\{P_\pm(d)\}_{0<d<1/3}$. 
We show that each of $\Gamma_\pm$, say $\Gamma_+$, intersects any circle of Apollonius with foci $\rho$ and $\rho^{-1}$, $C_A$, in exactly one point. 

Firstly, since $\Gamma_+$ is an open curve joining $\rho$ and a point $(2/3,0)$ on the real axis, it must intersect $C_A$. 
Secondly, if we put $\Omega_\mp=\Omega\cup\{\R z\gtrless 1/2\}$ \footnote{The image of $\Omega_+$ (or $\Omega_-$) of $\varphi$ is the intersection of $\mathcal{D}$ and the upper half plane (or the lower half plane respectively).}, then $\Gamma_+\subset\Omega_-$, and hence $\Gamma_+\cap C_A=\Gamma_+\cap(C_A\cap\Omega_-)$. 
Let $P_+(d)=(x(d),y(d))$, then, as $x(d)$ is an increasing function and $y(d)$ a decreasing function of $d$, we have $dy/dx<0$ on $\Gamma_+$, whereas 
$C_A\cap\Omega_-$ can be expressed as a graph of an increasing function. Therefore $\Gamma_+$ intersects $C_A$ in at most one point. 

We remark that the statement can also be proved by computation showing that $|\varphi(P_\pm(d))|$ is a monotonely increasing function of $d$ with $|\varphi(P_\pm(0))|=0$ and $\lim_{d\to1/3}|\varphi(P_\pm(d))|=1$. 
{\hfill{\small{${\square}$}}\par\medskip}

\bigskip
\noindent
{\bfseries Proof of Lemma {\rm \bf \ref{lemma}.}} \ \ 
Suppose $\Delta_{ba}$ is expressed by a complex number $z=z_{ba}=x+yi$ $(x,y\in\mathbb{R})$ in the fundamental domain $\Omega$. 
Since $|z|^2=a^2$ and $|z-1|^2=b^2$, $x$ and $y$ satisfy
\begin{equation}
\begin{array}{rcl}
x&=&\displaystyle \frac12(a^2-b^2+1), \\[2mm]
y^2&=&\displaystyle \frac14\left(b^2-(a-1)^2\right)\left((a+1)^2-b^2\right).
\end{array}
\label{xy_ab}
\end{equation}
Let $w=f(z)$ and $\theta=\arg w$. Then 
\[
\tan\theta=\frac{\I\left(-\rho\,\frac{z-\rho}{z-\rho^{-1}}\right)}{\R\left(-\rho\,\frac{z-\rho}{z-\rho^{-1}}\right)}
=\sqrt3\,\frac{x^2+y^2-2x}{x^2+y^2+2x-2}.
\]
Substitution of \eqref{xy_ab} gives
\[
\tan\theta=\sqrt3\,\frac{b^2-1}{2a^2-b^2-1}.
\]
Since $\Delta_{ab}$ is a mirror image of $\Delta_{ba}$, we have $\varphi(z_{ab})=\overline{\varphi(z_{ba})}$, and since $\varphi(z)={\left(f(z)\right)}^3$, we have 
\[
\arg\varphi(z_{ab})=\arg\varphi(z_{ba})+6\theta \>\>\>\mbox{modulo }\>2\pi.
\]

On the other hand, $T_q$ acts on $\mathcal{D}$ as a rotation by $6\tau_q$, where $\tau_q=\arg\left(-1+(1-2q)\sqrt3\,i\right)$. 
Suppose 
\begin{equation}\label{tan_theta=tan_sigma}
-(1-2q)\sqrt3=\sqrt3\,\frac{b^2-1}{2a^2-b^2-1}. 
\end{equation}
It means $\tan\tau_q=\tan\theta$, which implies $\tau_q\equiv\theta$ modulo $\pi$ and hence $6\tau_q\equiv6\theta$ modulo $2\pi$, which implies $T_q(\varphi(z_{ab}))=\varphi(z_{ba})$, i.e., $T_q([\Delta_{ab}])=[\Delta_{ba}]$. 

The equation \eqref{tan_theta=tan_sigma} gives
\[
q=\frac{1-a^2}{b^2+1-2a^2}.
\]
The other two values follow from the above by Corollary \ref{cor_T_q_equal}. 
{\hfill{\small{${\square}$}}\par\medskip}

\bigskip
\noindent
{\bfseries Proof of Theorem {\rm \bf \ref{thm_rationality}.}} \ \ 
When $[\Delta]=[\Delta']$ the conclusion follows from Corollary \ref{cor_T_q_id}. 
When $[\Delta]\ne[\Delta']$ the conclusion follows from Proposition \ref{prop_isosceles} for isosceles triangles and from Corollaries \ref{cor_right}, \ref{cor_SAP} and Lemma \ref{lemma} for right and SAP triangles. 
{\hfill{\small{${\square}$}}\par\medskip}

\section{Compass and straightedge constructibility}
%
%
\begin{proposition}\label{constructibility}
Given two triangles $\Delta=\triangle ABC$ and $\Delta'=\triangle A'B'C'$. 
Whether $\Delta$ is equisectionally equivalent to $\Delta'$ or not can be determined using a straightedge and compass, and if the answer is affirmative, a real number $q$ such that $T_q([\Delta])=[\Delta']$ is compass-and-straightedge constructible. 
\end{proposition}

\begin{proof}
One can construct the following with a compass and straightedge in the following order: 
\begin{enumerate}
\item Two points $D$ and $\overline{D}$ ($D\in\Pi_A$) such that both $\triangle DBC$ and $\overline{D}BC$ are regular triangles. 
\item  A vertex $\widehat A'$ in $\Pi_A$ such that $[\triangle \widehat A'BC]=[\triangle A'B'C']$. 
\item  An oriented circle $\Gamma(D, A, \overline{D})$ and another oriented circle $\Gamma(D, \widehat A', \overline{D})$ (see Definition \ref{def_Gamma}). 
\item  A circle of Apollonius $C_A$ with foci $D$ and $\overline{D}$ that passes through $A$, since the center is the intersection of a bisector of the edge $BC$ and a tangent line of $\Gamma(D,\overline{D}, A)$ at point $A$. 
\item  A decision whether $\Delta\sim\Delta'$ or not, since the answer is affirmative if and only if $\widehat A'\in C_A$.  
\end{enumerate}

\medskip
Assume $\Delta\sim\Delta'$ in what follows. 

\medskip
\begin{enumerate}\setcounter{enumi}{5}
\item The signed angle $\theta$ ($-2\pi/3<\theta<2\pi/3$) at point $D$ from the oriented circle $\Gamma(D, A, \overline{D})$ to $\Gamma(D, \widehat A', \overline{D})$. 
\item  At least one $n\in\mathbb{Z}$ such that 
\[
\frac\theta2+\frac{n\pi}3\in\left(\frac\pi2,\frac{3\pi}2\right)\quad\mbox{ modulo }2\pi. 
\]
\item  The value $q$ which is given by 
\begin{equation}\label{f_q_constructibility}
q=\frac12\left[1+\frac1{\sqrt3}\tan\left(\frac\theta2+\frac{n\pi}3\right)\right],
\end{equation}
where $n$ is given by the proceeding step. 
\end{enumerate}

\medskip
We explain the process (6),(7),(8). By working in the fundamental domain $\Omega$, we may assume that $B=0, C=1, A=z$, and $\widehat A'=z'$. By formulae \eqref{def_varphi} and \eqref{formula_T_q}, we want a real number $q$ such that 
\begin{equation}\label{f_q_z_z'}
6\arg\left(-1+(1-2q)\sqrt3\,i\,\right)
=3\left(\arg\left(\frac{z'-\rho}{z'-\rho^{-1}}\right)-\arg\left(\frac{z-\rho}{z-\rho^{-1}}\right)\right) \>\>\>\textrm{modulo}\>2\pi.
\end{equation}
Since the right hand side divided by $3$ is equal to the singed angle $\hat\theta$ at $0$ from the oriented line $\Gamma(0,(z-\rho)/(z-\rho^{-1}),\infty)$ to the oriented line $\Gamma(0,(z'-\rho)/(z'-\rho^{-1}),\infty)$, 
and a linear fractional transformation $f^{-1}$ maps $0, \,(z-\rho)/(z-\rho^{-1}), \>(z'-\rho)/(z'-\rho^{-1})$, and $\infty$ to $\rho, z, z'$, and $\rho^{-1}$ respectively, the signed angle $\hat\theta$ is equal to the signed angle $\theta$ at point $\rho$ from the oriented circle $\Gamma(\rho, z,\rho^{-1})$ to the oriented circle $\Gamma(\rho, z',\rho^{-1})$. 

If $q$ is given by \eqref{f_q_constructibility}, then 
\[
-1+(1-2q)\sqrt3\,i=-1-i\tan\left(\frac\theta2+\frac{n\pi}3\right),
\]
and therefore, $6\arg\left(-1+(1-2q)\sqrt3\,i\,\right)\equiv3\theta\>$ (modulo $2\pi$), which means \eqref{f_q_z_z'}. 
\end{proof}

\bigskip

Jun O'Hara

Department of Mathematics and Informatics,Faculty of Science, 

Chiba University

1-33 Yayoi-cho, Inage, Chiba, 263-8522, JAPAN.  

E-mail: ohara@math.s.chiba-u.ac.jp

\end{document}